\documentclass[12pt,a4paper]{article}

\usepackage{amsfonts}
\usepackage{amsmath}
\usepackage{amssymb}
\usepackage{amstext}
\usepackage{amsthm}
\usepackage{mathrsfs}

\newtheorem{teo}{Theorem}

\newtheorem{lemma}[teo]{Lemma}
\newtheorem{prop}[teo]{Proposition}

\theoremstyle{remark}
\newtheorem{rem}[teo]{Remark}

\DeclareMathOperator{\Id}{Id}
\DeclareMathOperator{\essinf}{essinf}

\newcommand{\R}{\mathbb R}
\newcommand{\eps}{\varepsilon}
\newcommand{\D}{{\cal D}^{1,2}}
\newcommand{\Ne}{{\cal N}}

\begin{document}
\title{Existence of minimal nodal solutions for the
Nonlinear Schroedinger equations with $V(\infty)=0$}
\author{M.Ghimenti\thanks{Dipartimento di Matematica Applicata,
Universit\`a di Pisa,via Bonanno 25b, 56100, Pisa, Italy},
A.M.Micheletti\thanks{Dipartimento di Matematica Applicata,
Universit\`a di Pisa,via Bonanno 25b, 56100, Pisa, Italy}}
\maketitle
\begin{abstract}
\noindent We consider the problem $\Delta u+V(x)u=f'(u)$ in
$\R^N$.
Here the nonlinearity has a double power behavior and
$V$ is invariant under an orthogonal involution, with $V(\infty)=0$.
An existence theorem of one pair of solutions which change
sign exactly once is given.

\noindent{\bf Mathematics Subject Classification:} 35J64, 35J20

\noindent{\bf Keywords:} Nonlinear Equations, Variational Methods,
Orlicz Spaces, Sign Changing Solutions
\end{abstract}

\section{Introduction}

It is well known that stationary states of Nonlinear Schroedinger
equations lead to problems of the
type

\begin{equation}\label{PV}
\left\{
\begin{array}{ll}
\tag{$\mathscr P$}
-\Delta u+V(x)u=f'(u),&x\in\R^N;\\
E_V(u)<\infty.
\end{array}
\right.
\end{equation}

where the energy functional is defined by
\begin{equation}
E_V(u)=\frac{1}{2}\int_{\R^N}|\nabla u|^2dx+
\frac{1}{2}\int_{\R^N}V(x)u^2(x)dx-\int_{\R^N}f(u)dx.
\end{equation}

We consider a function $f\in C^2(\R,\R)$ even
with $f(0)=f'(0)=0$ that
satisfies the following requirements:
\begin{description}
\item[$(f_1)$] there exists $\mu>2$ such that
\begin{equation}\label{fmu}
0<\mu f(s)\leq f'(s)s<f''(s)s^2 \text{ for all } s\neq 0
\end{equation}
\item[$(f_2)$]
there exist positive numbers $c_0,c_2,p,q$ with $2<p<2^*<q$ such that

\begin{equation}\label{f1}
\left\{
\begin{array}{ll}
c_0|s|^p\leq f(s)&\text{ for }|s|\geq1;\\
c_0|s|^q<f(s)&\text{ for }|s|\leq1;
\end{array}
\right.
\end{equation}
\begin{equation}
\left\{
\begin{array}{ll}
|f''(s)|\leq c_2|s|^{p-2}&\text{ for }|s|\geq1 \\
|f''(s)|\leq c_2|s|^{q-2}&\text{ for }|s|\leq1 \\
\end{array}
\right.
\end{equation}
where $2^*=\frac{2N}{N-2}$.
\end{description}

We assume $V\in L^{N/2}(\R^N)$ and
\begin{equation}\label{condV}
||V||_{L^{N/2}}<S:=\inf_{u\in{\D}}\frac{\int_{\R^N}|\nabla u|^2}
{\left(\int_{\R^N}|u|^{2^*}\right)^{2/2^*}}
\end{equation}

In the case of a single-power nonlinearity some paper
has been devoted to the existence of positive solutions when
potential $V$ vanishes at infinity. Among others we recall \cite{AFM05,BC90} and
we quote the references therein.

In pioneering work Berestycki and Lions \cite{BL83a,BL83b}
showed the existence of a
positive solution in the case $V\equiv0$ when $f''(0)=0$, $f$ has a
supercritical growth near the origin and subcritical at infinity.

More recently in the papers \cite{BR,BGM04,BGM04b,BM04,Pis} the double-power
growth condition $(f_2)$ has been used to obtain the existence of positive solutions for
different problems of the tipe (\ref{PV}).
In particular, in \cite{BGM04} the authors proved that if $V\geq 0$ and $V>0$ on a set of positive
measure the problem (\ref{PV}) has no ground state solution, i.e. there is no solution
$u$ of (\ref{PV}) which minimizes the functional $E_V$ on the Nehari manifold $\Ne_V$,
defined by
\begin{equation}
\Ne_V=\left\{u\in\D\::\ \langle\nabla E_V(u),u\rangle=0, u\neq 0\right\}.
\end{equation}
On the contrary there exists a ground state solution either if $V\leq0$ and
$V<0$ on a set of positive measure, or $V\equiv0$.

In this paper we are interested in the existence of sign changing solutions.
Besides the difficulty posed by the lack of compactness we have another problem:
there is  no natural regular constraint for sign changing solution of problem
(\ref{PV}). To overcome this difficulty we consider the problem

\begin{equation}\label{PVtau}
\left\{
\begin{array}{ll}
\tag{$\mathscr P_\tau$}
-\Delta u+V(x)u=f'(u),&x\in\R^N;\\
E_V(u)<\infty;\\
u(\tau x)=-u(x),
\end{array}
\right.
\end{equation}
where $\tau$ is a non trivial orthogonal involution that is a linear orthogonal
transformation on $\R^N$ such that $\tau\neq\Id$ and $\tau^2=\Id$ ($\Id$
being the identity on $\R^N$).

We assume $V(\tau(x))=V(x)$.

By the nontrivial orthogonal involution $\tau$ on $\R^N$ we can define a
self adjoint linear isometry on $\D$ which we also denote $\tau$. We define
\begin{equation}
\begin{array}{l}
\tau:\D\rightarrow\D;\\
(\tau u)(x):=-u(\tau(x)).
\end{array}
\end{equation}
If $u(\tau x)=-u(x)$, it will be called $\tau$-antisymmetric. Note that non trivial
antisymmetric solutions are changing sign or nodal solutions. Nodal solutions which
change sign exactly once will be called minimal nodal solutions.

We define
\begin{eqnarray*}
\Ne_V^\tau=\Ne_V\cap\D_\tau&\text{where}&\D_\tau=\{u\in\D\ :\ \tau u=u\}.
\end{eqnarray*}

The non trivial antisymmetric solutions of (\ref{PVtau}) are the
critical points of $E_V$ on $\Ne_V^\tau$

We set now
\begin{eqnarray}
\mu_V=\mu_V(\R^N):=\inf_{\Ne_V} E_V;&&
\mu_0=\mu_0(\R^N):=\inf_{\Ne_0} E_0;\\
\mu_V^\tau=\mu_V^\tau(\R^N):=\inf_{\Ne_V^\tau} E_V;&&
\mu_0^\tau=\mu_0^\tau(\R^N):=\inf_{\Ne_0^\tau} E_0.
\end{eqnarray}

We shall prove the following results.
\begin{teo}\label{mainVpos}
If $V(x)>0$ for almost every $x$,
then
\begin{displaymath}
\mu_V^\tau=\mu_0^\tau=2\mu_0,
\end{displaymath}
and $\mu_V^\tau$ is not achieved. Then the problem (\ref{PVtau}) has
no solution of minimal energy.
\end{teo}

We consider the following class of potentials.

\begin{equation}\label{defVy}
V_y(x)=
\left\{
\begin{array}{cll}
a|x-y|-1&&|x-y|<1;\\
a|x-\tau y|-1&&|x-\tau y|<1;\\
0&&\text{elsewhere}
\end{array}
\right.
\end{equation}
where $a\in\R$ is chosen such that $||V||_{L^{N/2}}<S$, $S$ as in (\ref{condV}).
We can prove the following existence result.

\begin{teo}\label{mainVneg}
For the potential $V_y$ such that $|y-\tau y|$ is sufficiently large we have that
$\mu_V^\tau<\mu_0^\tau$  and it is achieved. Then the problem
(\ref{PVtau}) has at least one
pair of antisymmetric solutions which change sign exactly once, and the energy of these
solutions is minimal.
\end{teo}

We want to mention some recent work about sign changing solutions.
The existence of a sequence of
nodal solutions and some properties for the number of their nodal
domains has been obtained in \cite{BW03} considering the problem in a bounded
smooth domain $\Omega$ with $V\equiv0$, and in \cite{BLW04}
in $\R^N$ with $\essinf V>0$.

In \cite{CW04} there is a theorem of multiplicity of solutions for the problem
$-\Delta u+Vu=q(x)|u|^{p-2}u$ where $V(x)$ and $q(x)$ tend to some
positive number $V_\infty$ and $q_\infty$ respectively as $|x|\rightarrow\infty$.
However no precise information is given whether there are sign changing solutions or not.
If $V\equiv1$ and $q(x)$ suitable chosen, with $||q-1||_{\infty}$ small,
Hirano \cite{Hi01} prove the existence of at least two pairs
of sign changing solutions.

In \cite{CC03} the equation $-\Delta u+\lambda u=|u|^{p^*-2}u$, $\lambda>-\lambda_1$ on a
symmetric domain is considered and the effect of the domain topology on the number of minimal
nodal solutions in studied.

The plan of the paper is the following.

In section \ref{sezvar} we recall some technical result concerning
the appropriate function spaces required by the growth properties of $f$; the
proof of these results are contained in \cite{BF04,BGM04,BM04}.
In section \ref{sezsplitting} we prove a {\em splitting lemma} which is a variant of
a well known result of \cite{St84}; this lemma is the ingredient to handle
the problem with lack of compactness. In section \ref{sezmain} we prove our results.

We will use the following notations
\begin{itemize}
\item $u^+=\max(0,u)$; $u^-=\min(0,u)$;
\item {$\D=\D(\R^N)=$ completion of $C_0^\infty(\R^N)$ with respect to the norm
$$||u||_{\D}=\left(\int_{\R^n}|\nabla u|^2\right)^{1/2};$$}
\item $v_y(x)=v(x+y)$;
\item $\displaystyle G_V(u)=\langle\nabla E_V(u),u\rangle=\int|\nabla u|^2+V u^2-f'(u)u$
\item
$\displaystyle
g_u(t)=g^V_u(t):=E_V(tu)=\int\limits_{\R^N}\frac{t^2}{2}(|\nabla u|^2+Vu^2)-f(tu)dx$
\item $B_R=\{x\in\R^N\ :\ |x|<R \}$;
\item $B_R(z)=B(z,R)=\{x\in\R^N\ :\ |x-z|<R \}$;
\item $A^C=\R^N\smallsetminus A$, where $A\subset\R^N$.
\end{itemize}
\section{Variational Setting}\label{sezvar}

In order to study the functional $E_V$, by the growth assumption on $f$,
is useful to consider the functional space $L^p+L^q$,
where $2<p<2^*<q$.
Hereafter we recall some result contained in
\cite{BF04,BGM04,BM04}.

Given $p\neq q$, we consider the space $L^p+L^q$ made up of the
functions $v:\R^N\rightarrow\R$ such that
\begin{equation}
v=v_1+v_2\ \text{ with }\ v_1\in L^p, v_2\in L^q.
\end{equation}
The space $L^p+L^q$ is a Banach space equipped with the norm:
\begin{equation}
||v||_{L^p+L^q}=\inf\{\ ||v_1||_{L^p}+||v_2||_{L^q}\ :\ v_1\in L^p, v_2\in L^q,\
v_1+v_2=v\}.
\end{equation}
It is well known (see, for example \cite{BL76}) that $L^p+L^q$
coincides with the dual of $L^{p'}\cap L^{q'}$. Then:
\begin{equation}
L^p+L^q=\left(L^{p'}\cap L^{q'} \right)'\ \text{ with }\ p'=\frac p{p-1},\ q'=\frac q{q-1}.
\end{equation}
We recall some results useful for this paper.
\begin{rem}\label{lp+lq}
Set $\Gamma_u=\{u\in\R^N\ :\ |u(x)|\geq1 \}$. We have
\begin{enumerate}
\item if $v\in L^p+L^q$, the following inequalities hold:
\begin{eqnarray*}
&&\max\left[||v||_{L^q(\R^N\smallsetminus\Gamma_v)}-1,
\frac{1}{1+|\Gamma_v|^\frac1\tau}||v||_{L^p(\Gamma_v)}\right]\leq\\
&\leq&||v||_{L^p+L^q}\leq\\
&\leq&\max[||v||_{L^q(\R^N\smallsetminus\Gamma_v)},||v||_{L^p(\Gamma_v)}]
\end{eqnarray*}
when $\tau=\frac{pq}{q-p}$;
\item let $\{v_n\}\subset L^p+L^q$. Then $\{v_n\}$ is bounded in $L^p+L^q$ if and only
if the sequences $\{|\Gamma_{v_n}|\}$ and
$\{||v||_{L^q(\R^N\smallsetminus\Gamma_{v_n})}+||v||_{L^p(\Gamma_{v_n})}\}$
are bounded.
\item  we have
$L^{2^*}\subset L^p+L^q$ when $2<p<2^*<q$. Then, by Sobolev inequality,
we get the continuous embedding
\begin{equation*}
\D(\R^N)\subset L^p+L^q.
\end{equation*}

\end{enumerate}
\end{rem}

\begin{rem}\label{stimef}
If $f$ satisfies the hypothesis that we have made in the previous section, we have that
\begin{enumerate}
\item $f'$ is a bounded map from $L^p+L^q$ into $L^{p/p-1}\cap L^{q/q-1}$;
\item $f''$ is a bounded map from $L^p+L^q$ into $L^{p/p-2}\cap L^{q/q-2}$;
\item $f''$ is a continuous  map from $L^p+L^q$ into $L^{p/p-2}\cap L^{q/q-2}$.
\end{enumerate}
\end{rem}

At last we recall some result on Nehari manifolds.
For the proofs we refer to \cite{BGM04,BM04}.

\begin{rem}The functional $E_V$ is of class $C^2$ and it holds
\begin{equation}
\langle \nabla E_V(u),v\rangle=\int\nabla u\nabla v+Vuv-f'(u)vdx.
\end{equation}
Moreover the Nehari manifold defined as
\begin{equation}
\Ne_V=\left\{u\in\D\smallsetminus 0\  :\ \int|\nabla u|^2+Vu^2-f'(u)udx=0\right\}
\end{equation}
is of class $C^1$ and its tangent space at the point $u$ is
\begin{equation}
T_u{\Ne_V}=\left\{u\in\D:\int2\nabla u\nabla v+2Vuv-f'(u)vdx-f''(u)uv=0\right\}.
\end{equation}
\end{rem}
\begin{rem}\label{normanehari}
We have
\begin{equation}
\inf\limits_{u\in\Ne_V}||u||_{\D}>0.
\end{equation}
\end{rem}
\begin{proof}
At first notice that, by \ref{condV}
\begin{equation}
\exists c>0\text{ s.t }\int|\nabla u|^2+Vu^2\geq c||u||^2_{\D}, \ \forall u\in\D.
\end{equation}
Now, let $\{u_n\}$ a minimizing sequence in $\Ne_V$. By contradiction, we suppose that
$u_n$ converges to 0. We set $t_n=||u_n||_{\D}$, hence we can write $u_n=t_nv_n$
where $||v_n||_{\D}=1$. By claim 3 of Remark \ref{lp+lq}, the sequence is bounded in
$L^p+L^q$. Since $u_n\in\Ne_V$ and $t_n$ converges to 0, we have
\begin{eqnarray*}
&&ct_n=\frac c{t_n}||u_n||_{\D}^2\leq
\frac 1{t_n}\int\limits_{\R^N}|\nabla u_n|^2+Vu_n^2=
\int\limits_{\R^N} f'(t_nv_n)v_n\leq\\
&&\leq c_1t_n^{q-1}\int\limits_{\R^N\smallsetminus \Gamma_{t_nv_n}}|v_n|^q+
c_1t_n^{p-1}\int\limits_{\Gamma_{t_nv_n}}|v_n|^p\leq\\
&&\leq c_1t_n^{q-1}\int\limits_{\R^N\smallsetminus \Gamma_{t_nv_n}}|v_n|^q+
c_1t_n^{p-1}\int\limits_{\Gamma_{v_n}}|v_n|^p\leq\\
&&\leq c_1t_n^{q-1}\int\limits_{\R^N\smallsetminus \Gamma_{v_n}}|v_n|^q+
c_1t_n^{q-1}\int\limits_{(\R^N\smallsetminus \Gamma_{t_nv_n})\cap\Gamma_{v_n}}
\frac{|v_n|^p}{t_n^{q-p}}
+c_1t_n^{p-1}\int\limits_{\Gamma_{v_n}}|v_n|^p\leq\\
&&\leq c_1t_n^{q-1}\int\limits_{\R^N\smallsetminus \Gamma_{v_n}}|v_n|^q+
2c_1t_n^{p-1}\int\limits_{\Gamma_{v_n}}|v_n|^p.
\end{eqnarray*}
Hence we get
\begin{displaymath}
c\leq c_1t_n^{q-2}\int\limits_{\R^N\smallsetminus \Gamma_{v_n}}|v_n|^q+
2c_1t_n^{p-2}\int\limits_{\Gamma_{v_n}}|v_n|^p
\end{displaymath}
and by claim 2 of Remark \ref{lp+lq} we get the contradiction.
\end{proof}
\begin{rem}\label{lemma4.1}
We have that for any given $u\in{\D}\smallsetminus \{0\}$, there exists a unique real number
$t^V_u > 0$ such that $t^V_uu\in\Ne_V$ and $E_V(t^V_uu)$ is the maximum for the
function
\begin{displaymath}
t\rightarrow E_V(tu),\ \  t > 0.
\end{displaymath}
\end{rem}
\begin{proof}
Given $u\neq0$ we set, for $t\geq0$:
\begin{equation}
g_u(t)=g^V_u(t):=E_V(tu)=\int\limits_{\R^N}\frac{t^2}{2}(|\nabla u|^2+Vu^2)-f(tu)dx.
\end{equation}
We have:
\begin{eqnarray}
g'_u(t)&=&\int\limits_{\R^N}t|\nabla u|^2+Vtu^2-u f'(tu)dx;\\
g''_u(t)&=&\int\limits_{\R^N}|\nabla u|^2+Vu^2-u^2 f''(tu)dx.
\end{eqnarray}
By hypothesis on $f$, if $g'_u(\bar t)=0$ we have
\begin{equation}
\bar t^2 g''_u(\bar t)=\int\limits_{\R^N}\bar tu f'(\bar tu)-\bar t^2u^2 f''(\bar tu)dx<0,
\end{equation}
then $\bar t$ is a maximum point for $g_u$. Furthermore $0 = g_u(0)=g'_u(0)$ and
$g_u''(0) > 0$ by the hypothesis on $V$, then 0 is a local minimum point for $g_u$.
By (\ref{f1}), for $t\geq1$, we
have
\begin{eqnarray}
\nonumber
g_u(t)&\leq&\int\limits_{\R^N}\frac{t^2}{2}(|\nabla u|^2+Vu^2)dx-
c_0\!\!\!\!\int\limits_{\{|tu\leq1|\}}|tu|^qdx-
c_0\!\!\!\!\int\limits_{\{|tu>1|\}}|tu|^pdx\leq\\
&\leq&\int\limits_{\R^N}\frac{t^2}{2}(|\nabla u|^2+Vu^2)dx-
c_0\!\!\!\!\int\limits_{\{|tu>1|\}}|tu|^pdx\leq\\
\nonumber
&\leq&\frac{t^2}{2}\int\limits_{\R^N}(|\nabla u|^2+Vu^2)dx-
c_0t^p\!\!\!\!\int\limits_{\{|u>1|\}}|u|^pdx;
\end{eqnarray}
the last quantity diverges negatively as $t\rightarrow\infty$, since $p > 2$, and the claim
follows.
\end{proof}
We search antisymmetric solutions of (\ref{PVtau}). To do that, we
look for  critical points of the restriction of $E_V$ to $\Ne_V^\tau$.
In fact, if $\bar u\in\Ne_V^\tau$ is a critical point of the
restriction of $E_V$ to $\Ne_V^\tau$, then
\begin{eqnarray*}
E'_V(\bar u)\varphi=\langle\nabla E_V(\bar u),\varphi\rangle=0&&
\forall\varphi\in T_{\bar u}\Ne_V\cap\D_\tau(\R^N).
\end{eqnarray*}
But
$\nabla E_V(\bar u)=\tau\nabla E_V(\tau\bar u)=\tau\nabla E_V(\bar u)$,
so we can see that
$\nabla E_V(\bar u)=0$.

\section{A splitting lemma}\label{sezsplitting}

We recall that a sequence $\{u_n\}_n\in\D$ such that
$E_V(u_n)\rightarrow c$,
and there exists a sequence $\eps_n\rightarrow0$ s.t.
$|E_V'(u_n)\varphi|\leq \eps_n||\varphi||$, for all $\varphi\in\D$
is a Palais-Smale sequence at level $c$ for $E_V$.

In the same way we say that $\{u_n\}_n\in\Ne_V^\tau$ such that
$E_V(u_n)\rightarrow c$,
and there exists a sequence $\eps_n\rightarrow0$ s.t.
$|E_V'(u_n)\varphi|\leq \eps_n||\varphi||$, for all
$\varphi\in T_{u_n}\Ne_V \cap\D_\tau$
is a Palais-Smale sequence at level $c$ for $E_V$ restricted to $\Ne_V^\tau$.

A functional $f$ satisfies the $(PS)_c$ condition if all the
Palais-Smale sequences at level $c$ converge.

Unfortunately the functional $E_V$ on $\Ne_V^\tau$ does not satisfy the {\em PS} condition in
all the energy range. The  following lemma provides a description of the {\em PS} sequences
in $\D_\tau$.

This splitting lemma is a
fundamental tool to obtain the claimed results.
The main idea of this result spread over a result of M. Struwe \cite{St84} that described
all the {\em PS} sequences for $E_V$ on $H^1(\Omega)$ when $f(u)=u|u|^{2^*-2}$ and
$V(x)=\lambda u$
\begin{lemma}\label{splitting}
Let $\{u_n\}_n$ a PS sequence at level $c$
for the functional $E_V$ restricted to the manifold $\Ne^\tau_V$.
Then, up to a subsequence, there exist two integers $k_1,k_2\geq0$,
$k_1+k_2$ sequences $\{y_n^j\}_n$,
an antisymmetric solution $u^0$ of the problem $-\Delta u +Vu=f'(u)$,
and $k_1$ solutions $u^j$, $j=1,\dots k_1$, and $k_2$ antisymmetric solutions
$u^j$, $j=k_1+1,\dots, k_1+k_2$, of the problem
$-\Delta u =f'(u)$ such that
\begin{enumerate}
\item if $j=1,\dots,k_1$, then $\tau y_n^j\neq y_n^j$, and $|y_n^j|\rightarrow\infty$ as $n\rightarrow\infty$
\item if $j=k_1+1,\dots,k_2$, then $\tau y_n^j=y_n^j$, and $|y_n^j|\rightarrow\infty$ as $n\rightarrow\infty$
\item {$u_n(x)=u^0(x)+\sum\limits_{j=1}^{k_1}[u^j(x-y_n^j)+\tau u^j(x-y_n^j)]+
\sum\limits_{j=k_1+1}^{k_2}u^j(x-y_n^j)+o(1)$}
\item $E_V(u_n)\rightarrow E_V(u^0)+2\sum\limits_{j=1}^{k_1} E_0(u^j)+\sum\limits_{j=k_1+1}^{k_2}E_0(u^j)$
\end{enumerate}
\end{lemma}
\begin{proof}
Since $u_n$ is a {\em PS} sequence for the functional $E_V$ restricted to the manifold $\Ne_V$, then $u_n$ is a
{\em PS} sequence for the functional $E_V$. For Step 1 of \cite[Lemma 3.3]{BGM04} we get that
$u_n$ converges to $u^0$ weakly in $\D$ (up to subsequence) and $u^0$ solves $-\Delta u +Vu=f'(u)$.

Since $\tau u_n=u_n$, we have $\tau u^0=u^0$. In fact, if $u_n\stackrel{\D}{\rightharpoonup}u^0$, then, for
every $R>0$, we have that $u_n\stackrel{L^2(B_R)}{\longrightarrow}u^0$, so $u_n(x)\rightarrow u^0(x)$
for almost all $x$.

We set
\begin{equation*}
\psi_n(x)=u_n(x)-u^0(x).
\end{equation*}
Then $\tau\psi_n=\psi_n$, and
$\psi_n\rightharpoonup0$ weakly in $\D$. By Steps 2 and 4 of \cite[Lemma 3.3]{BGM04}, we get that $\psi_n$ is
a {\em PS} sequence for $E_0$. If $\psi_n\nrightarrow0$ strongly in $\D$, then by Steps 3 and 4 of
\cite[Lemma 3.3]{BGM04} we get that
there exists a sequence $\{\xi_n\}\subset\R^N$ with $|\xi_n|\rightarrow \infty$ as $n\rightarrow \infty$,
and $\psi_n(x+\xi_n)\rightharpoonup u^1(x)$ where $u^1$ is a weak solution of $-\Delta u =f'(u)$.

We consider in $\R^N=\Gamma \oplus \Gamma^\bot$, where $\Gamma:=\{x \in \R^N\ :\ \tau(x)=x \}$. We consider
$P_\Gamma$ the projection on the subspace $\Gamma$. At this point we must distinguish two cases
\begin{description}
\item[Case I:] if $|\xi_n-\tau(\xi_n)|$ is bounded we define $y_n=P_\Gamma \xi_n$.
\item[Case II:] if $|\xi_n-\tau(\xi_n)|$ is unbounded we define $y_n=\xi_n$.
\end{description}

\noindent{\bf Case I}
In this case there exist a solution $\widetilde u\in{\D_\tau}\smallsetminus \{0\}$ of
$-\Delta u=f'(u)$ and a {\em PS} sequence $\{\widetilde\psi_n\}_n$ for $E_0$ such that
\begin{equation}
\psi_n=\widetilde{\psi}_n+\widetilde u(x+y_n),
\end{equation}
$\widetilde\psi_n\rightharpoonup0$ weakly in $\D$, and
\begin{equation}
E_0(\widetilde\psi_n)=E_0(\psi_n)-E_0(\widetilde u)+o(1).
\end{equation}

We can assume, without loss of generality, that $\xi_n=P_\Gamma\xi_n+w$,
where $w\in\Gamma^\bot$. We now
consider the sequence $\{\psi_n(x+y_n)\}_n$ which is bounded: hence,
up to subsequence $\{\psi_n(x+y_n)\}_n$
converges to $\widetilde u(x)$ weakly in $\D$, $\tilde u(x)=u^1(x-w)\neq0$, then
$-\Delta \widetilde u=f'(\widetilde u)$. Furthermore, because
$\tau y_n=y_n$ we have that $\tau\widetilde u=\widetilde u$. We define
\begin{equation}
\widetilde \psi_n(x):=\psi_n(x)-\widetilde u(x-y_n).
\end{equation}
We will verify that $\widetilde \psi_n$ is a {\em PS} sequence for $E_0$.
Indeed by Lemma 2.11 of \cite{BM04}
we get
\begin{eqnarray*}
E_0(\widetilde \psi_n)&=&E_0(\psi_n(x)-\widetilde u(x-y_n))=
E_0(\psi_n(x+y_n)-\widetilde u(x))=\\
&=&E_0(\psi_n)-E_0(\widetilde u)+o(1),
\end{eqnarray*}
and, because $\{\psi_n\}_n$ is a {\em PS} sequence for $E_0$, we have that $E_0(\psi_n)$
converges, so $E_0(\widetilde \psi_n)$
converges, also.

Again, since $\{\psi_n\}$ is a {\em PS} sequence, we have that exists an
$\varepsilon_n\rightarrow  0$ such that,
for all $\varphi\in C^\infty_0$
\begin{eqnarray*}
&&\Big|(E_0)'(\widetilde\psi_n)[\varphi]\Big|=
\left|\int\nabla\psi_n\nabla\varphi-\int\nabla\widetilde u_{-y_n}\nabla\varphi-
\int f'(\psi_n-\widetilde u_{-y_n}\varphi)\right|=\\
&&=\!\left|\!\int\!\!\left[f'(\psi_n(x))-f'(\widetilde u(x{-y_n}))-
f'(\psi_n(x)-\widetilde u(x{-y_n}))\right]\!\varphi(x) dx\!+
\!\varepsilon_n||\varphi||\right|\!\!=\\
&&=\!\left|\!\int\!\!\left[f'(\psi_n(z\!+\!y_n))-f'(\widetilde u(z))-
f'(\psi_n(z\!+\!y_n)-\widetilde u(z))\right]\!\varphi_{y_n}\!(z)dz\!+
\!\varepsilon_n||\varphi||\right|.
\end{eqnarray*}
Now we can choose an $R>0$ and split this integral as follows.
\begin{eqnarray*}
\Big|(E_0)'(\widetilde\psi_n)[\varphi]\Big|&=&
\Big|\int_{B_R}[f'(\psi_n(x+y_n))-f'(\widetilde u)]\varphi_{y_n}(z)dz+\\
&&+\int_{\R^N\smallsetminus B_R}[f'(\psi_n(x+y_n))-f'(\psi_n(x+y_n)-\widetilde u)]\varphi_{y_n}dz-\\
&&-\!\!\!\int_{\R^N\smallsetminus B_R}\!\!\!f'(\widetilde u)\varphi_{y_n}\!
-\!\!\!\int_{B_R}\!\!f'(\psi_n(x+y_n)\!-\!\widetilde u)\varphi_{y_n}dz\!+
\!\varepsilon_n||\varphi||\Big|\!\!\leq\\
&\leq&A_n\ \gamma_{n,R}\ ||\varphi||_{\D}+B_n\ M_R\ ||\varphi||_{\D}+\varepsilon_n||\varphi||_{\D}
\end{eqnarray*}
where
\begin{eqnarray*}
&&A_n=||f''(\theta\psi_n(\cdot+y_n)+(1-\theta)\widetilde u)-
f''(\theta\psi_n(\cdot+y_n)-\theta\widetilde u)||_{L^{p/p-2}};\\
&&\gamma_{n,R}=||\psi_n(\cdot-y_n)-\widetilde u||_{L^p(B_R)};\\
&&B_n=||f''(\psi_n(\cdot+y_n)+\theta\widetilde u)-
f''(\theta\widetilde u)||_{L^{p/p-2}\cap L^{q/q-2}};\\
&&M_R=||\widetilde u||_{L^p+L^q(\R^N\smallsetminus B_R)},
\end{eqnarray*}
for some $0<\theta<1$.

By \cite[Lemma 2.11]{BM04} we have that both $A_n$ and $B_n$ are bounded.
Since $M_R\rightarrow 0$
as $R\rightarrow+\infty$ and, given $R$, $\gamma_{n,R}\rightarrow0$ as
$n\rightarrow+\infty$, we get the claim.

\noindent{\bf Case II} In this case there exist a solution $u^1\neq 0$ of
$-\Delta u=f'(u)$ and a {\em PS} sequence
$\{\widetilde\psi_n\}_n\subset \D_\tau$ for $E_0$ such that
\begin{eqnarray}
&&\psi_n(x)=\widetilde\psi_n(x)+u^1(x-y_n)-u^1(\tau x -y_n)+o(1);\\
&&E_0(\widetilde\psi_n)=E_0(\psi_n)-2E_0(u^1)+o(1)\label{caso2cond2}.
\end{eqnarray}
We define $\widetilde\psi_n=\psi_n-\gamma_n$,
\begin{equation}
\gamma_n(x)=u^1(x-y_n)\chi\left(\frac{|x-y_n|}{\rho_n}\right)-
u^1(\tau x-y_n)\chi\left(\frac{|x-\tau y_n|}{\rho_n}\right),
\end{equation}
where $\rho_n:=\frac{|y_n-\tau y_n|}{2}\rightarrow\infty$ for
$n\rightarrow\infty$, and, as usual,
$\chi:\R^+_0\rightarrow[0,1]$ is a $C^\infty$ function such that
$\chi(s)\equiv0$ for all $s\geq2$, $\chi(s)\equiv1$ for all
$s\leq1$ and $|\chi'(s)|\leq2$ for all $s$.

It is trivial that $\tau \gamma_n=\gamma_n$, so
$\tau\widetilde\psi_n=\widetilde\psi_n$. Furthermore, easily we have
\begin{equation*}
\psi_n(x)=\widetilde\psi_n(x)+u^1(x-y_n)-u^1(\tau x -y_n)+o(1).
\end{equation*}
Now we have to prove (\ref{caso2cond2}), and to show that
$\widetilde \psi_n$ is a {\em PS} sequence.

At first we prove that
\begin{equation}\label{normagamma}
||\widetilde\psi_n||^2_{\D}=||\psi_n-\gamma_n||^2=||\psi_n||^2+2||u^1||^2+o(1).
\end{equation}
In fact we have that $||\psi_n-\gamma_n||^2=||\psi_n||^2+||\gamma_n^2||^2-2(\psi_n,\gamma_n)_{\D}$,
and it is easy to see
that $||\gamma_n^2||^2\rightarrow 2||u^1||^2$. Furthermore
\begin{eqnarray}\label{prodscal}
(\psi_n,\gamma_n)_{\D}&=&\int \nabla\psi_n
\nabla\left(u^1(x-y_n)\chi\left(\frac{|x-y_n|}{\rho_n}\right)\right)+\\
\nonumber &&+
\int \nabla\psi_n
\nabla\left(u^1(\tau x-y_n)\chi\left(\frac{|\tau x-y_n|}{\rho_n}\right)\right),
\end{eqnarray}
and the first term converges to $\int|\nabla u^1|$. For the second term we have
\begin{eqnarray*}
&&\int \nabla\psi_n
\nabla\left(u^1(\tau x-y_n)\chi\left(\frac{|\tau x-y_n|}{\rho_n}\right)\right)
=\\
&&=\int \left(\nabla\psi_n
\nabla u^1(\tau x-y_n)\right)\chi(\cdot)
+\int \left(\nabla\psi_n
\nabla \chi(\cdot)\right)u^1(\tau x-y_n).
\end{eqnarray*}
The last term of the equation vanishes when $n\rightarrow\infty$, while,
remembering that $\psi_n$ is symmetric, and setting $z=\tau x-y_n$, we have
\begin{eqnarray*}
&&\int \left(\nabla\psi_n(x)
\nabla u^1(\tau x-y_n)\right)\chi(\cdot)dx=\\
&&\int \left(\tau\nabla\psi_n(z+y_n)
\nabla u^1(z)\right)\chi\left(\frac{z}{\rho_n}\right)dz\rightarrow-\int|\nabla u^1|^2,
\end{eqnarray*}
so we have proved (\ref{normagamma}).

We want now to estimate
$$\int f(\widetilde\psi_n)=\int f(\psi_n-\gamma_n).$$
Set
\begin{eqnarray*}
I_1&=&\int_{|x-y_n|<2\rho_n}f\left(\psi_n(x)-u^1(x-y_n)\chi\left(\frac{|x-y_n|}{\rho_n}\right)\right);\\
I_2&=&\int_{|\tau x-y_n|<2\rho_n}f\left(\psi_n(x)+u^1(\tau x-y_n)\chi\left(\frac{|\tau x-y_n|}{\rho_n}\right)\right);\\
I_3&=&\int_{\{B(y_n)_{2\rho_n}\cup B(\tau y_n)_{2\rho_n}\}^C}f(\psi_n(x)),
\end{eqnarray*}
we have
\begin{equation}
\int f(\widetilde\psi_n)=I_1+I_2+I_3.
\end{equation}
We have that
\begin{equation}
\left[\psi_n(z+y_n)-u^1(z)\chi\left(\frac{|z|}{\rho_n}\right)\right]
\chi\left(\frac{|z|}{\rho_n}\right)\rightharpoonup0\text{ in }\D.
\end{equation}

By \cite[Lemma 2.11]{BM04}, then we have that
\begin{eqnarray*}
I_1&=&\int_{|z|<2\rho_n}f\left(\psi_n(z+y_n)-u^1(z)\chi\left(\frac{|z|}{\rho_n}\right)\right)=\\
&=&\int_{|z|<2\rho_n}f(\psi_n(z+y_n))-
\int_{|z|<2\rho_n}f\left(u^1(z)\chi\left(\frac{|z|}{\rho_n}\right)\right)+o(1)=\\
&=&\int_{|z|<2\rho_n}f(\psi_n(z+y_n))-
\int_{\R^n}f\left(u^1(z)\right)+o(1).
\end{eqnarray*}
In the same way, because $\psi_n$ is symmetric,
\begin{equation}
\left[\psi_n(\tau z+\tau y_n)-u^1(z)\chi\left(\frac{|z|}{\rho_n}\right)\right]
\chi\left(\frac{|z|}{\rho_n}\right)\rightharpoonup0\text{ in }\D,
\end{equation}
and
\begin{eqnarray*}
I_2=\int_{|\tau x-y_n|<2\rho_n}f(\psi_n(x))-
\int_{\R^n}f\left(u^1(x)\right)+o(1).
\end{eqnarray*}
At last we have
\begin{eqnarray}\label{fpsi}
\nonumber
\int f(\widetilde\psi_n)&=&
\int\limits_{|x-y_n|<2\rho_n}f(\psi_n(x))+\int\limits_{|\tau x-y_n|<2\rho_n}f(\psi_n(x))+\\
&&
\!\!+\!\!\int\limits_{\{B(y_n)_{2\rho_n}\cup B(\tau y_n)_{2\rho_n}\}^C}
\!\!\!\!\!\!\!\!\!\!\!\!f(\psi_n(x))-
2\!\!\int\limits_{\R^n}f(u^1(x))\!+\!
o(1)=\\
\nonumber&=&\int_{\R^N} f(\psi_n(x))-2\int_{\R^n}f(u^1(x))+o(1).
\end{eqnarray}
From (\ref{normagamma}) and (\ref{fpsi}) we obtain, as claimed
\begin{equation}
E_0(\widetilde\psi_n)=E_0(\psi_n-\gamma_n)=E_0(\psi_n)-2E_0(u^1)+o(1);
\end{equation}
furthermore, because $\{\psi_n\}_n$ is a {\em PS} sequence for $E_0$, we have that
$E_0(\widetilde\psi_n)\rightarrow c$ for some $c\in\R$.

To complete the proof we must show that
\begin{equation}
|(E_0)'(\widetilde\psi_n)\varphi|\leq\varepsilon_n||\varphi||_{\D},
\end{equation}
where $\varepsilon_n\rightarrow0$.
Set
\begin{eqnarray*}
I^1_n&=&\!\!\!\int\limits_{|x-y_n|<2\rho}\!\!\!\!\left[
f'(\psi_n(x))-f'\left(\psi_n(x)-u^1(x-y_n)\chi\!\left(\frac{|x-y_n|}{\rho_n}\!\right)\!\right)\!
\right]\varphi(x)dx-\\
&&-\int\nabla\left(u^1(x-y_n)\chi\left(\frac{|x-y_n|}{\rho_n}\right)\right)\nabla\varphi(x);\\
I^2_n&=&\!\!\!\!\!\int\limits_{|\tau x-y_n|<2\rho}\!\!\!\!\!\left[
f'(\psi_n(x))\!-\!f'\!\left(\psi_n(x)-u^1(\tau x-y_n)\chi\!\left(\frac{|\tau x-y_n|}{\rho_n}\!\right)\!\right)\!
\right]\!\!\varphi(x)dx\!-\\
&&-\int\nabla\left(u^1(\tau x-y_n)\chi\left(\frac{|\tau x-y_n|}{\rho_n}\right)\right)\nabla\varphi(x);\\
I^3_n&=&\int\limits_{\{B(y_n)_{2\rho_n}\cup B(\tau y_n)_{2\rho_n}\}^C}
[\nabla\psi_n-f'(\psi_n)]\varphi,
\end{eqnarray*}
we have that
\begin{equation}
|(E_0)'(\widetilde\psi_n)\varphi|=I^1_n+I^2_n+I^3_n.
\end{equation}
Immediately we have that $I^3_n\leq\varepsilon_n||\varphi||$; furthermore,
we can estimate $I^1_n$ as before,
obtaining
\begin{eqnarray*}
I^1_n&=&\!\!\int\limits_{|z<2\rho_n|}\!\!\!\left[
f'(\psi_n(z+y_n))\!-\!f'\left(\psi_n(z+y_n)-u^1(z)\chi\!\left(\!\frac{|z|}{\rho_n}\!\right)\!\right)\!
\right]\!\!\varphi(z+y_n)dz-\\
&&-\int\limits_{|z<2\rho_n|}f'(u^1\chi)\varphi(z+y_n)\ dz+\varepsilon_n||\varphi||.
\end{eqnarray*}
Setting
\begin{equation}
\alpha_n(z):=\psi_n(z+y_n)-u^1(z)\ \chi\left(\frac{|z|}{\rho_n}\right),
\end{equation}
we have
\begin{eqnarray*}
|I^1_n|&=&\Big|\int\limits_{B_{\rho_n}}\!\!\!\left[
f'(u^1\chi+\alpha_n)-f'(\alpha_n)-f'(u^1\chi)
\right]\varphi_{y_n}dz+\varepsilon_n||\varphi||
\ \Big|,
\end{eqnarray*}
and, chosen an $R>0$,
\begin{eqnarray*}
|I^1_n|&=&\Big|\int\limits_{B_{\rho_n}}\!\!\!\left[
f'(u^1\chi+\alpha_n)-f'(u^1\chi)\right]\varphi_{y_n}dz+\\
&&+\int\limits_{(\R^N\smallsetminus B_R)\cap B_{2\rho_n}}
\left[
f'(u^1\chi+\alpha_n)-f'(\alpha_n)\right]\varphi_{y_n}dz-\\
&&-\int\limits_{(\R^N\smallsetminus B_R)\cap B_{2\rho_n}}
f'(u^1\chi)\varphi_{y_n}dz-\int\limits_{B_{\rho_n}}f'(\alpha_n)\varphi_{y_n}dz+\varepsilon_n||\varphi||\ \Big|.
\end{eqnarray*}
Using that $f'(0)=0$ at last we have
\begin{eqnarray*}
&&|I^1_n|\leq||f''(u^1\chi+\theta_1\alpha_n)-f''(\theta_1\alpha_n)||_{L^{p/p-2}(\R^N)}
||\varphi||_{L^p(\R^N)}||\alpha_n||_{L^p(B_R)}+\\
&&+||\!f''\!(\alpha_n\!\!+\!\theta u^1\chi)\!\!-\!\!
f''\!(\theta_1u^1\chi)||\!_{L^{\frac{p}{p-2}}\cap L^{\frac{q}{q-2}}(\R^N\!)}
||\varphi||_{L^p+L^q}
||u^1\chi||_{L^p+L^q(B_R^C)}\!\!+\\
&&+\varepsilon_n||\varphi||,
\end{eqnarray*}
where $0<\theta,\theta_1<1$. By Remark \ref{stimef} we get
\begin{equation}
|I^1_n|\leq\varepsilon_n||\varphi||_{\D}.
\end{equation}
In the same way we can estimate $|I^2_n|$, and this concludes the proof.
\end{proof}

\begin{rem}
If $u_n\in\Ne_V^\tau$ is a Palais-Smale sequence of the restriction of
$E_V$ to $\Ne_V^\tau$, that is $E_V(u_n)$ converges and
\begin{displaymath}
|E'_V(u_n)w|\leq\varepsilon_n||w||_{\D}\ \ \forall w\in T_{u_n}\Ne_V\cap\D_\tau,
\end{displaymath}
where $\varepsilon_n\rightarrow0$, then $u_n$ is a Palais-Smale sequence for
the functional $E_V$.
\end{rem}

This remark, combined with the splitting lemma, provides a complete description of
the {\em PS} sequences in our case.

\section{The main result}\label{sezmain}

At this point we prove some technical lemmas.

Let $u\in\Ne_V$, then $u^+$ and $u^-$ belong to $\Ne_V$.
Furthermore, if $u$ is antisymmetric, we have $E_V(u^+)=E_V(u^-)$. So,
if $u\in\Ne_V^\tau$, we get
\begin{displaymath}
E_V(u)=E_V(u^+)+E_V(u^-)=2E_V(u^+)\geq 2\inf_{\Ne_V}E_V=2\mu_V.
\end{displaymath}
This implies that
\begin{eqnarray}
\mu_V^\tau&\geq&2\mu_V;\\
\mu_0^\tau&\geq&2\mu_0.
\end{eqnarray}

\begin{rem}
We have $\mu_0^\tau=2\mu_0$
\end{rem}
\begin{proof}
We have to proof that $\mu_0^\tau\leq 2\mu_0$. It is possible to find a sequence
$\{u_k\}_k\subset \Ne_0^\tau$ such that $E_0(u_k)\rightarrow 2\mu_0$. So
\begin{displaymath}
\mu_0^\tau\leq \inf\limits_k E_0(u_k)\leq 2 \mu_0.
\end{displaymath}
The construction of $\{u_k\}_k$ is quite similar to the construction of $\{z_k\}_k$
in the next theorem. So also the proof that $E_0(u_k)\rightarrow 2\mu_0$. Therefore,
for the sake of simplicity, we omit the detailed proof of this result.
\end{proof}

We are ready now to prove the main lemma of this section
\begin{lemma}\label{main}
We have that $\mu^\tau_V\leq\mu^\tau_0$
\end{lemma}
\begin{proof}
We prove it by steps

\noindent{\bf Step I} We know that $w$ exists
such that $\mu_0=E_0(w)$. Let $\chi(x)$ a smooth, real function such that
\begin{displaymath}
\chi=\left\{
\begin{array}{ll}
1& B(0,1);\\
0& \R^N\smallsetminus B(0,3).
\end{array}
\right.
\end{displaymath}
We also ask that $\chi(x)=\chi(|x|)$ and that $|\nabla\chi|\leq1 $.

Let $\{y_k\}\subset\R^N$ s.t. $|y_k|\rightarrow\infty$ and
$|\tau(y_k)-y_k|\rightarrow\infty$. Let $\rho_k$ be defined as
\begin{displaymath}
\rho_k:=\frac{|\tau(y_k)-y_k|}{6}.
\end{displaymath}

At last we define a function in $\D$
\begin{equation}
z_k=z_k^1+z_k^2,
\end{equation}
where
\begin{eqnarray}
z_k^1&=&w(x-y_k)\chi\left(\frac{x-y_k}{\rho_k}\right);\\
z_k^2&=&-w(\tau(x)-y_k)\chi\left(\frac{x-\tau(y_k)}{\rho_k}\right).
\end{eqnarray}
Obviously we have that $\tau z_k^1=z_k^2$ and $\tau z_k^2=z_k^1$, so
$z_k\in\D_\tau \ \forall k$. Furthermore
$z_k^1$ and $z_k^2$ have disjoint supports, so
\begin{equation}
E_V(t\cdot z_k)=E_V(t\cdot z_k^1)+E_V(t\cdot z_k^2)\ \forall t>0.
\end{equation}
We know, from Remark \ref{lemma4.1}, that it exists a $t_k>0$ s.t.
$t_k\cdot z_k^1\in\Ne_V$. It's easy to see that,
for such $t_k$ we have that $t_k\cdot z_k^2\in\Ne_V$ and
$t_k\cdot z_k\in\Ne_V^\tau$.

In the next we will prove that
$E_V(t_kz_k)\rightarrow2\mu_0$, when $k\rightarrow\infty$.

\noindent{\bf Step II} We prove that $||z_k^1(x)-w(x-y_k)||_{\D}\rightarrow0$ for
$k\rightarrow\infty$.

Set $w_k:=w(x-y_k)$, and
$\gamma_k:=\left(1-\chi\left(\frac{x-y_k}{\rho_k}\right)\right)$,
we have
\begin{eqnarray*}
||z_k^1(x)-w_k||_{\D}^2&=&\int\left|\nabla\left[ \gamma_k w_k\right]\right|^2
\leq2\int\gamma_k^2|\nabla w_k|^2
+2\int|\nabla\gamma_k|^2w_k^2\leq\\
&\leq&\int\limits_{|x-y_k|>3\rho_k}|\nabla w_k|^2+
\frac{2}{\rho_k^2}\int\limits_{\rho_k<|x-y_k|<3\rho_k}w_k^2\rightarrow0,
\end{eqnarray*}
as $k\rightarrow\infty$.

\noindent{\bf Step III} We prove that it exists $c,C>0$ such that
$c<t_k<C$ for all $k$.

By Remark \ref{normanehari}, we know that,
if $t_kz_k^1\in\Ne_V$, then it exists $M>0$ such that, for all
$k$, $M\leq||t_kz_k^1||$.
Furthermore, for the above step
\begin{equation}\label{normazk}
||z_k^1||\rightarrow||w_k||=||w||.
\end{equation}
This implies that $c$ exists such that $t_k>c>0$ for all $k$.

For the other inequality we must prove that
\begin{equation}\label{Vw2}
\int V(x) w_k^2(x)dx\rightarrow 0 \text{ when }k\rightarrow\infty,
\end{equation}
in fact, fixed an $R>0$ we have
\begin{eqnarray*}
\left|
\int Vw_k^2
\right|&\leq&\int\limits_{B_R}|V|w_k^2+\int\limits_{\R^N\smallsetminus B_R}|V|w_k^2\leq\\
&\leq&||V||_{L^{\frac N2}(B_R)}||w_k||^2_{L^{2^*}(B_R)}+
||V||_{L^{\frac N2}(\R^N\smallsetminus B_R)}||w_k||^2_{L^{2^*}(\R^N)}.
\end{eqnarray*}
We have that $||V||_{L^{\frac N2}(\R^N\smallsetminus B_R)}\rightarrow0$ as
$R\rightarrow\infty$;
furthermore
\begin{equation}
\int\limits_{B_R}w^{2^*}(x-y_k)dx=\int\limits_{B_R(-y_k)}w^{2^*}(s)ds
\leq\int\limits_{\R^N\smallsetminus B_{(|y_k|-R)}}w^{2^*}(s)ds,
\end{equation}
thus $||w_k||^2_{L^{2^*}}\rightarrow0$ as $k\rightarrow\infty$,
that proves (\ref{Vw2}).

Now set a function
\begin{equation}
g_{z_k^1}^V(t):=E_V(tz_k^1).
\end{equation}
Obviously $g_w^0(t)=E_0(tw)=\frac{1}{2}t^2\int|\nabla w|^2-\int f(tw)$.
For Remark \ref{lemma4.1} we
know that there exists a $\bar t$ such that
$g^0_{w_{k}}(\bar t)=g^0_w(\bar t)<0$ for all $k$.
We want to prove that, for $k$ sufficiently big, we have
also $g^V_{z_k^1}(\bar t)<0$.
\begin{eqnarray*}
g^V_{z_k^1}(\bar t)-g^0_w(\bar t)&=&g^V_{z_k^1}(\bar t)-g^0_{w_k}(\bar t)=
\frac{1}{2}\int |\nabla \bar t z_k^1|^2+\frac{1}{2}\int V(x)\bar t^2 (z_k^1)^2-\\
&&-\int f(\bar t z_k^1) -
\frac{1}{2}\int |\nabla \bar t w_k|^2+\int f(\bar t w_k)=\\
&=&\frac{\bar t^2}{2}
\left[\int\! |\nabla z_k^1|^2\!-\!|\nabla w_k|^2\!+\!V(x)(z_k^1)^2\right]\!-\!\!
\int \!f(\bar t z_k^1) \!-\!f(\bar t w_k).
\end{eqnarray*}
By (\ref{normazk}) and by (\ref{Vw2}) we have
that the first integral of the right hand side of the equation vanishes when $k\rightarrow\infty$.
We estimate the last term.
\begin{eqnarray*}
\int f(\bar t z_k^1) -f(\bar t w_k)&=&
\int_{|x-y_k|>\rho_k} f\left(\bar t w_k \chi\left( \frac{x-y_k}{\rho_k}\right)\right) -f(\bar t w_k)=\\
&=&\int_{|s|>\rho_k} f\left(\bar t w \chi\left( \frac{s}{\rho_k}\right)\right) -f(\bar t w)=\\
&=&\!\!\int\limits_{|s|>\rho_k}\!\!\!f'\!\left(\left[\theta \chi\left( \!\frac{s}{\rho_k}\!\right)\!+\!
(1-\theta)\right]\bar t w\right)\left(\chi\left(\! \frac{s}{\rho_k}\!\right)\!- \!1\right)\bar t w =\\
&\leq&\bar t\int\limits_{|s|>\rho_k}\!\!wf'\left(\left[\theta \chi\left(\! \frac{s}{\rho_k}\!\right)\!+\!
(1-\theta)\right]\bar t w\right)\left( \chi\left( \!\frac{s}{\rho_k}\!\right)\!-\! 1\right),
\end{eqnarray*}
so
\begin{equation*}
\left|\int f(\bar t z_k^1) -f(\bar t w_k)\right|\leq
\bar t\!\!\int\limits_{|s|>\rho_k}\!\left|f'\left(\left[\theta \chi\left( \frac{s}{\rho_k}\right)+
(1-\theta)\right]\bar t w\right)\right|\left|  \chi\left( \frac{s}{\rho_k}\right) -1 \right|w
\end{equation*}
and the last term vanishes
when $k\rightarrow \infty$. In fact, for Remark \ref{stimef}, $f'$ is bounded in
$L^{\frac{p}{p-1}}\cap L^{\frac{q}{q-1}}$ and
$\left|  \chi\left( \frac{s}{\rho_k}\right) -1 \right|w\rightarrow 0$ in $L^{2^*}$.
So, for $k_0$ big enough, we have that
\begin{displaymath}
\exists c,C\in\R^+\ \text{s.t. } 0<c<t_k<C \ ,\ \forall k>k_0
\end{displaymath}

\noindent{\bf Step IV} We want to prove that $E_V(t_kz_k^1)\rightarrow\mu_0$.
We have
\begin{eqnarray*}
|E_V(t_kz_k^1)-E_V(w_k)|&=&|E_V'(\theta t_kz_k^1+(1-\theta)w_k)(z_k^1-w_k)|.
\end{eqnarray*}
We know, for Step 2, that $||z_k^1-w_k||_{\D}\rightarrow0$.

Furthermore
\begin{equation}
||\theta t_kz_k^1+(1-\theta)w_k||\leq ||z_k^1||t_k+||w_k||
\end{equation}
that is bounded because $t_k$ is bounded and by Step II.

At this point by Remark \ref{stimef} we get the claim.

We know also that
\begin{eqnarray*}
E_V(w_k)-\mu_0&=&E_V(w_k)-E_0(w)=E_V(w_k)-E_0(w_k)=\\
&=&\int V(x)w_k^2\stackrel{k\rightarrow \infty}{\longrightarrow}0
\end{eqnarray*}
for (\ref{Vw2}). Then
\begin{eqnarray}
|E_V(t_kz_k^1)-\mu_0|&\leq&|E_V(t_kz_k^1)-E_V(w_k)|+|E_V(w_k)-\mu_0|\rightarrow0
\end{eqnarray}
as we wanted to prove.

\noindent{\bf Conclusion}
We know that $t_kz_k\in\Ne_V^\tau$. Then
$$
E_V(t_kz_k)\geq\mu_V^\tau:=\inf_{u\in{\Ne_V^\tau}}E_V(u).
$$
Hence
\begin{displaymath}
\mu_V^\tau\leq E_V(t_kz_k)=E_V(t_kz^1_k)+E_V(t_kz^2_k)\rightarrow2\mu_0=\mu_0^\tau
\end{displaymath}
that gives us the proof.
\end{proof}

We are ready, now, to prove the first result claimed in the introduction.

\begin{proof}[Proof of Theorem \ref{mainVpos}]
First, we prove that
\begin{equation}
\forall u\in\Ne_0\  \exists t_u^Vu\in\Ne_V\ \text{ s.t }\  E_V(t_u^Vu)\geq E_0(u).
\end{equation}
In fact, by Remark \ref{lemma4.1}, we have that
for every $u\in\Ne_0$, there exist $t^V_u > 0$ such that
$t^V_u u\in\Ne_V$. Then we have:
\begin{eqnarray*}
0&=&g'_u(t^V_u u)=\langle \nabla E_0(t^V_uu),u\rangle+t^V_u\int\limits_{\R^N}Vu^2.
\end{eqnarray*}
Since $V > 0$ we have that $\int Vu^2 > 0$ and $\langle \nabla E_0(t^V_uu),u\rangle < 0$.
Hence $t_u^V>t_u^0=1$.
Let us observe that by (\ref{fmu})
the function $s\rightarrow\int \frac 12 f'(su)su-f(su)dx$ is strictly increasing, then,
remembering that $t^V_uu\in\Ne_V$, we have:
\begin{eqnarray*}
E_V(t^V_uu)&=&\frac 12 \int f'(t^V_uu)t^V_uu-f(t^V_uu)dx\geq\\
&\geq&\frac 12 \int f'(u)u-f(u)dx=E_0(u).
\end{eqnarray*}

If $u\in\Ne_0^\tau$, we can prove in the same way that $t_u^Vu\in\Ne_V^\tau$ and that
\begin{equation}
E_V(t_u^Vu)\geq E_0(u)\geq \mu_0^\tau.
\end{equation}
So
\begin{equation}
\inf_{w\in\Ne_V^\tau}E_V(w)=\inf_{u\in\Ne_0^\tau}E_V(t_u^Vu)\geq E_0(u)\geq \mu_0^\tau.
\end{equation}
Theorem \ref{main} provides us the other inequality.

Suppose now that there exists $v\in \Ne_V^\tau$ such that
$\mu_V^\tau =E_V(v)$. We know that $\int V(x)v(x)^2 > 0$ and
\begin{displaymath}
0 = \langle\nabla E_0(v), v\rangle + \int V(x)v(x)^2 dx,
\end{displaymath}
so, consequently  $\langle\nabla E_0(v), v\rangle< 0$.
Then, by Remark \ref{lemma4.1}, we get $t_v^0<t_v^V=1$.
As said before, the function
$s\mapsto \int \frac 12 f '(sv) sv - f (sv) dx$ is
strictly increasing, so we have
\begin{displaymath}
E_0(vt_v^0)=\int\frac 12 f'(t^0_v)t^0_vv- f (t^0_vv)dx <
\int\frac 12 f'(v)v - f(v)dx=E_V(v)=\mu_V^\tau
\end{displaymath}
and we get a contradiction.

\end{proof}

Now we prove the following preliminary result.

\begin{prop}
There exists a class of potential $V(x)$ such that
$\mu_V^\tau<\mu_0^\tau$.
\end{prop}
\begin{proof}
We consider the class of potentials defined in (\ref{defVy}).
We want to show that, when $|y|\rightarrow\infty$ and $|y-\tau y|\rightarrow\infty$,
then $\mu_{V_y}^\tau<\mu_0^\tau$.
We prove it by steps.

Take $w\in\Ne_0$ such that $E_0(w)=\mu_0$, $w$ radially symmetric and $w>0$
(see \cite{BM04,BL83a,BL83b}). By means of
$w$, we define
\begin{equation}
z_y(x)=w(x-y)-w(x-\tau y).
\end{equation}
\noindent{\bf Step I }
We prove that, for $|y- \tau y|\rightarrow\infty$,
\begin{equation*}
g_{z_y}^{V_y}(t)\rightarrow t^2\int|\nabla w(x)|^2dx
+at^2\int[|x|-1]w^2(x)dx-2\int f(tw(x))dx,
\end{equation*}

where
\begin{equation}
g_{z_y}^{V_y}(t)=E_{V_y}(tz_y)=
\int\limits_{\R^N}\frac{t^2}{2}(|\nabla z_y|^2+V_yz_y^2)-f(tz_y)dx.
\end{equation}
Now
\begin{equation}
\frac{t^2}{2}\!\!\int|\nabla z_y|^2=
\frac{t^2}{2}\!\!\int|\nabla w(x-y)|^2+
|\nabla w(\tau x-y)|^2+
\nabla w(x-y)\nabla w(x-\tau y).
\end{equation}
After a change of variables, the first two terms are equals to
$\frac{t^2}{2}\int|\nabla w|^2$, and the last term vanishes.
So we have that, for all $t$,
\begin{eqnarray}
\frac{t^2}{2}\!\!\int|\nabla z_y|^2\rightarrow
t^2\int|\nabla w|^2&&\text{ when }|y-\tau y|\rightarrow\infty.
\end{eqnarray}
In a similar way consider
\begin{eqnarray}
\nonumber
\int V_yz_y^2&=&a\int\limits_{|x-y|<1}(|x-y|-1)z_y^2+
a\int\limits_{|x-\tau y|<1}(|x-\tau y|-1)z_y^2=\\
\label{Vy}
&=&a\int\limits_{|x-y|<1}(|x-y|-1)[w(x-y)-w(x-\tau y)]^2+\\
\nonumber&&+a\int\limits_{|x-\tau y|<1}(|x-\tau y|-1)[w(x-y)-w(x-\tau y)]^2.
\end{eqnarray}
By means of a change of variables we obtain
\begin{eqnarray*}
\int\limits_{|x-y|<1}(|x-y|-1)z_y^2&=&
\int\limits_{|s|<1}(|s|-1)[w(s)+w(s+y-\tau y)]^2.
\end{eqnarray*}
It is not difficult to prove that
\begin{eqnarray}
\int\limits_{|s|<1}(|s|-1)w^2(s+y-\tau y)\rightarrow0;&&\\
\int\limits_{|s|<1}(|s|-1)w(s)w(s+y-\tau y)\rightarrow0.&&
\end{eqnarray}
In the same way we proceed for
the second term of the (\ref{Vy}),
obtaining
\begin{equation}
\frac{t^2}{2}
\int V_vz_y^2\rightarrow at^2\int\limits_{|x|<1}(|x|-1)w^2(x)dx
\text{ when }|y|\rightarrow\infty.
\end{equation}
We have to estimate now $\int f(tw)$. Fixed an $R>0$, we have
\begin{equation}
\int f(tz_y)=\int\limits_{B_R(y)}f(tz_y)+\int\limits_{B_R(\tau y)}f(tz_y)+
\int\limits_{(B_R(y)\cup B_R(\tau y))^C}f(tz_y).
\end{equation}
For the first term we have
\begin{eqnarray*}
\int\limits_{B_R(y)}\!\!\!f(tz_y)&=&\int\limits_{B_R}f(t w(s)-t w(s+y-\tau y))=
\int\limits_{B_R}f(t w)+\\
&&+\!\!\int\limits_{B_R}\!f'(\theta t w(s)+(1-\theta)t w(s+y-\tau y))
[t w(s+y-\tau y)],
\end{eqnarray*}
for some $\theta \in [0,1]$.

Now, for Remark \ref{stimef}, we have that
$f'(\theta tw(x-y)+(1-\theta)tw(\tau x-y))$ is bounded in
$L^{p'}\cap L^{q'}$, in fact
$\theta tw(x-y)+(1-\theta)tw(\tau x-y)$ is bounded in $\D$ and so in $L^{p}+ L^{q}$.
Furthermore,
$w(s+y-\tau y)\rightarrow0$ strongly in $L^P(B_R)$ when $|y-\tau y|\rightarrow +\infty$.

Concluding we get
\begin{equation}
\int\limits_{B_R(y)}\!\!\!f(tz_y)=
\int\limits_{B_R}f(tw(s))ds+I_1(R,y),
\end{equation}
where, given $R>0$, $I_1(R,y)\rightarrow0$ when $|y-\tau y|\rightarrow\infty$.
In the same way we can conclude that
\begin{equation}
\int\limits_{B_R(\tau y)}\!\!\!f(tz_y)=
\int\limits_{B_R}f(tw(s))ds+I_2(R,y),
\end{equation}
where, again, given $R>0$, $I_2(R,y)\rightarrow0$ when $|y-\tau y|\rightarrow\infty$.

For the last term we have that there exist a $\theta \in [0,1]$ such that
\begin{eqnarray*}
\int\limits_{(B_R(y)\cup B_R(\tau y))^C}\!\!\!\!\!\!\!\!\!\!\!f(tz_y)\!\!&=&
\!\!\!\int\limits_{(B_R(y)\cup B_R(\tau y))^C}
\!\!\!\!\!\!\!\!\!\!\!f(t(w(x-y)))+\\
&&+\!\!\!\!\!\!\!\!\!\!\int\limits_{(B_R(y)\cup B_R(\tau y))^C}
\!\!\!\!\!\!\!\!\!\!\!\!\!\!\!
f'(\theta t(w(x\!-\!y))\!-\!\!(1\!-\!\theta) w(x\!-\!\tau y))t w(x\!-\!\tau y)=\\
&=&
\!\!\!\int\limits_{(B_R\cup B_R(y- \tau y))^C}
\!\!\!\!\!\!\!\!\!\!\!f(t(w(s)))+\\
&&+\!\!\!\!\!\!\!\!\!\!\int\limits_{(B_R\cup B_R(\tau y -y))^C}
\!\!\!\!\!\!\!\!\!\!\!
f'(\theta t(w(\xi+\tau y-y))-(1-\theta) w(\xi))t w(\xi).
\end{eqnarray*}
Now,
\begin{equation}
\left|\
\int\limits_{(B_R\cup B_R(\tau y -y))^C}
\!\!\!\!\!\!\!\!\!\!\!
f'(\cdot)tw(\xi)\
\right| \leq t ||f'(\cdot)||_{L^{p'}\cap L^{q'}(\R^N)}
||w||_{L^{p}+ L^{q}(\R^N\smallsetminus B_R)},
\end{equation}
and we use  that $||w||_{{\D}(\R^N\smallsetminus B_R)}$ goes to zero
when $R\rightarrow\infty$ and that
$||f'(\cdot)||_{L^{p'}\cap L^{q'}(\R^N)}$ is bounded  by Remark \ref{stimef}.

At this point we have that

\begin{eqnarray}
\int f(tz_y)\rightarrow2\int f(w)&&
\text{when } |y-\tau y|\rightarrow\infty,
\end{eqnarray}
and we get the claim.

\noindent{\bf Step II }
There exists a $\bar t<1$ such that
\begin{eqnarray}
t^{V_y}_{z_y}\rightarrow\bar t&&
\text{when } |y-\tau y|\rightarrow\infty,
\end{eqnarray}
where $t^{V_y}_{z_y}$ is the maximum point of $g^{V_y}_{z_y}(t)$.

We set
\begin{equation}
\varphi(t)=t^2\int|\nabla w(x)|^2dx
+at^2\int[|x|-1]w^2(x)dx-2\int f(tw(x))dx.
\end{equation}
By Remark \ref{lemma4.1} there exists a unique maximizer $\bar t>0$
for the function $\varphi(t)$.

We know that
\begin{equation}
g^0_w(t)=\frac 12 t^2\int\nabla w^2-\int f(tw)
\end{equation}
reach its maximum for $t=1$.
Thus the maximum of the function
\begin{equation}
\varphi(t)=2 g^0_w(t)+a t^2\int[|x|-1]w^2(x)dx
\end{equation}
is achieved for $\bar t $, with $0<\bar t<1$.

Given $t_1<\bar t<t_2$, we can choose a $\delta>0$ such that
\begin{equation}
\max\{\varphi(t_1),\varphi(t_2)\}+\delta<\varphi(\bar t)-\delta.
\end{equation}
By Step I, for $|y-\tau y|$ sufficiently large, we obtain
\begin{equation}
g_{z_y}^{V_y}(t_i)<\varphi(t_i)+\delta<\varphi(\bar t)-\delta<g_{z_y}^{V_y}(\bar t).
\end{equation}
By Remark \ref{lemma4.1}, we know that $g_{z_y}^{V_y}(t)$ has an unique maximum point
$t_{z_y}^{V_y}$, thus
we conclude that
\begin{equation}
t_1<t_{z_y}^{V_y}<t_2.
\end{equation}
Since $t_1$ and $t_2$ are arbitrarily chosen, we get the claim.

\noindent{\bf Step III }
For $|y- \tau y|$ sufficiently large we have
\begin{equation}
\mu_{V_y}^\tau<\mu_0^\tau.
\end{equation}

We know that
\begin{eqnarray}\label{energia}
E_{V_y}(t^{V_y}_{z_y}\ z_y)=g_{z_y}^{V_y}(t_{z_y})\rightarrow \varphi(\bar t)&&
\text{for }|y-\tau y|\rightarrow\infty,
\end{eqnarray}
in fact, for all $\eps>0$ we have that, for $|y-\tau y|$ sufficiently large,
\begin{eqnarray*}
|g_{z_y}^{V_y}(t_{z_y})-\varphi(\bar t)|&=&
|g_{z_y}^{V_y}(t_{z_y})-g_{z_y}^{V_y}(\bar t)+g_{z_y}^{V_y}(\bar t)-\varphi(\bar t)|\leq\\
&\leq&|g_{z_y}^{V_y}(t_{z_y})-g_{z_y}^{V_y}(\bar t)|+|g_{z_y}^{V_y}(\bar t)-\varphi(\bar t)|\leq \eps.
\end{eqnarray*}
By Step I the second term goes to zero when
$|y-\tau y|\rightarrow\infty$. By Step II, $t^{V_y}_{z_y}\rightarrow \bar t$, so,
arguing as in Step I, we get the claim.
We observe that
\begin{eqnarray*}
\varphi(\bar t) &=&
\bar t\int|\nabla w|^2+ a\bar t^2\int[|x|-1]w^2-2\int f(\bar t w)<\\
&<& 2E_0(\bar t w)<2\mu_0,
\end{eqnarray*}
because $\bar t<1$ and $E_0(w)=\mu_0$. By (\ref{energia}) we get
\begin{eqnarray}
\mu_{V_y}^\tau\leq E_{V_y}(t^{V_y}_{z_y}\ z_y)<2\mu_0=\mu_0^\tau
&&\text{for }|y-\tau y|\text{ large enough},
\end{eqnarray}
that concludes the proof
\end{proof}

Now we are ready to prove the second result claimed in the introduction.
\begin{proof}[Proof of theorem \ref{mainVneg}]
By the Splitting Lemma and the above Proposition, we get
the existence of a minimizer for $E_{V_y}$,
for the class of potential
$V_y$ defined by (\ref{defVy}), when $|y-\tau y|$ large enough.

Let $\omega$ be this minimizer.
We know that $\omega$ changes sign, because it is antisymmetric by construction.
We have to prove that
$\omega$ changes sign exactly once.
Suppose that the set $\{x\in\R^N\ :\ \omega(x)>0 \}$ has $k$ connected
components $\Omega_1, \dots, \Omega_k$.
Set
\begin{equation}
\omega_i=\left\{
\begin{array}{lll}
\omega(x)&&  x\in\Omega_i\cup\tau\Omega_i;\\
0&&\text{elsewhere}
\end{array}
\right.
\end{equation}
For all $i$, $\omega_i\in \Ne_{V_y}^\tau$.
Furthermore we have
\begin{equation}
E_{V_y}(\omega)=\sum\limits_iE_{V_y}(\omega_i),
\end{equation}
thus
\begin{equation}
\mu_{V_y}^\tau=E_{V_y}(\omega)=\sum\limits_{i=1}^k E_{V_y}(\omega_i)\geq k\mu_{V_y}^\tau,
\end{equation}
so $k=1$, that concludes the proof.
\end{proof}

\nocite{Ne60}
\nocite{BGM04b}
\nocite{ABC97}

\providecommand{\bysame}{\leavevmode\hbox to3em{\hrulefill}\thinspace}

\end{document}